\documentclass{amsart}

\usepackage{lipsum}
\usepackage{amsfonts}
\usepackage{graphicx}
\usepackage{amsmath}
\usepackage{amssymb}
\usepackage[english]{babel}
\usepackage{tikz-cd}
\usepackage{hyperref}

\newtheorem{proposition}{Proposition}
\newtheorem{corollary}{Corollary}
\newtheorem{definition}{Definition}

\begin{document}

\title{The Cofree Functor on $G$-Sets and Its Approximations}

\author{Frank Murphy-Hernandez}
\address{Facultad de Ciencias, UNAM, Mexico City}
\email{murphy@ciencias.unam.mx}


\date{\today}


\begin{abstract}
We study the cofree functor on the category of group actions and examine its categorical and combinatorial properties. Motivated by this construction, we introduce a family of functors associated with subgroups and develop their theory through a corresponding coreflective subcategory of group actions. Finally, we investigate the comonad in the category of sets induced by the cofree adjunction and prove that it classifies the groups.
\end{abstract}

\maketitle

\section{Introduction}

The interplay between free and cofree constructions is a recurring theme throughout category theory. A functor is called \emph{free} when it appears as the left adjoint of the underlying functor to sets; the paradigmatic example assigns to a set the free module over an associative ring. Dually, a \emph{cofree} functor is the right adjoint to the underlying functor. While free constructions are often encountered early in one's mathematical education, their cofree counterparts tend to be more elusive: they are typically harder to construct, less explicit, and their categorical significance is frequently underexplored.

The most celebrated instance of a cofree functor arises in the theory of coalgebras: the forgetful functor from $K$-coalgebras to $K$-vector spaces admits a right adjoint, whose existence and structure were established in the seminal works of Sweedler \cite{Sweedler1969Hopf} and Abe \cite{Abe1980Hopf}. 

The existence and study of right adjoints to forgetful functors is not confined to the classical contexts of Hopf algebras. Indeed, the problem of constructing and characterizing cofree objects arises in a variety of settings. For instance, in the theory of Hopf modules and their generalizations, the work of Caenepeel, Militaru, and Zhu \cite{Caenepeel1997DOI-HOPF} explores the right adjoints associated with Doi–Hopf modules, revealing deep connections with Frobenius-type properties. More recently, Zhou \cite{Zhou2019Doi-Hopf} extended these ideas to the setting of Hom–Hopf algebras, further illustrating the pervasiveness of cofree constructions in non-classical algebraic frameworks. In a different direction, Blumberg and Hill \cite{Blumberg2017The} studied the right adjoint to the equivariant operadic forgetful functor in the context of incomplete Tambara functors, a setting of fundamental importance in modern equivariant stable homotopy theory. Even in the realm of effect algebras and quantum logic, Jenča \cite{Jenča2015A} has examined categorical constructions that rely on the existence of suitable adjoints, demonstrating the broad applicability of these ideas beyond purely algebraic contexts. These diverse examples underscore the importance of developing a systematic categorical understanding of cofree functors, which motivates the more focused study of their properties in the concrete setting of group actions, as undertaken in this article.

In the setting of group actions, the forgetful functor $U: G\text{-}\mathbf{Set} \rightarrow \mathbf{Set}$ also admits a right adjoint, the \emph{cofree $G$-set} functor $C_G$, which sends a set $X$ to the set $X^G = \operatorname{Hom}(G,X)$, which is equipped with the left $G$-action given by translation on the domain using the inverse. Although this construction is classical and appears implicitly in the literature on adjoint functors and group cohomology, a systematic study of its fine categorical structure, its combinatorial consequences, and its natural approximations has remained largely absent.

This paper undertakes precisely such a study. Our first goal is to revisit the adjunction $U \dashv C_G$ and extract from it concrete enumerative information. By identifying the fixed points of the $G$-action on $C_G(X)$ with the constant functions, we apply Burnside's lemma to the cofree $G$-set and derive elegant number-theoretic divisibility results. In particular, we recover and contextualize the classical congruence $p \mid n^p + (p-1)n$ for primes $p$ (Proposition~\ref{prop:fermat_prime}), as well as its generalization to prime powers $p^m$ (Proposition~\ref{prop:fermat_prime_power}).

A central contribution of this work is the introduction of a family of \emph{approximations} to the cofree functor, parameterized by the subgroups of $G$. For $H \leq G$, we define the \emph{truncated cofree functor}
\[
C_H^G(X) = X^{G/H},
\]
whose elements are functions indexed by the left cosets of $H$ rather than by the entire group. While $C_H^G$ resembles the full cofree functor, it fails to be right adjoint to the global underlying functor. To understand its categorical status, we introduce the \emph{$H$-trace preradical} $\operatorname{Tr}_H$, defined as the union of the images of all equivariant maps from $G/H$ into a given $G$-set. We prove that $\operatorname{Tr}_H$ is an idempotent preradical (Propositions~\ref{prop:tr_preradical} and~\ref{prop:tr_idempotent}) and, more significantly, that it serves as a coreflector onto the full subcategory $G\text{-}\mathbf{Set}_H$ of \emph{$H$-generated $G$-sets} (Proposition~\ref{prop:tr_coreflector}).

For a normal subgroup $H \trianglelefteq G$, we establish the fundamental natural isomorphism
\[
C_H^G \cong \operatorname{Tr}_H \circ C_G
\]
(Proposition~\ref{prop:CHG_iso}), revealing that the truncated cofree functor associated with $H$ can be recovered by applying the $H$-trace to the full cofree $G$-set. Consequently, $C_H^G$ is identified as the right adjoint of the underlying set functor \emph{when restricted to the coreflective subcategory} of $H$-generated $G$-sets (Proposition~\ref{prop:CHG_adjoint}).  

Finally, we study the comonad $K_G = U_G C_G$ induced by the fundamental adjunction. We show that the assignment $G \mapsto K_G$ is functorial and, moreover, fully faithful (Propositions~\ref{prop:faithfulness} and~\ref{prop:fullness}), so that isomorphisms of groups correspond exactly to isomorphisms of comonads (Corollary~\ref{cor:reconstruction}). In particular, the cofree comonad structure determines the group up to isomorphism, yielding a full embedding of the groupoid of groups into the groupoid of comonads on $\mathbf{Set}$.

The paper is organized as follows. Section~\ref{sec:prelim} collects the necessary preliminaries on categories, preradicals, coreflective subcategories, and comonads. Section~\ref{sec:cofree} defines the cofree functor and establishes its adjointness to the underlying functor. Section~\ref{sec:burnside} applies Burnside's lemma to the cofree $G$-set to derive arithmetic consequences. Section~\ref{sec:approx} introduces the approximations via subgroups, the $H$-trace preradical, and the coreflective subcategory of $H$-generated $G$-sets. Section~\ref{sec:comonads} examines the induced comonads and their invariance under isomorphisms, and proves the reconstruction theorem.

\section{Preliminaries}
\label{sec:prelim}

We denote by $\mathbf{Set}$ the category of sets and maps. For a group $G$, we denote by $G\text{-}\mathbf{Set}$ the category of $G$-sets and equivariant maps. If $X$ is a $G$-set and $I$ is a set, we denote by $X^{(I)}$ the coproduct of $I$ copies of $X$. This set is precisely the disjoint union of $I$ copies of $X$.

For a group $G$ and a subgroup $H$ of $G$, the set of left cosets $G/H$ is naturally a $G$-set. Our references for group theory are \cite{aschbacher2000finite} and \cite{rotman2012introduction}.

A preradical on a category $\mathcal{C}$ is a subfunctor of the identity functor $\mathrm{Id}_{\mathcal{C}}$. Concretely, it assigns to each object $X$ a subobject $i^X_{r(X)} \colon r(X) \to X$, and for every morphism $f \colon X \to Y$, the restriction of $f$ to $r(X)$ factors through $r(Y)$; that is, $f \circ i^X_{r(X)} = i^Y_{r(Y)} \circ r(f)$. For the theory of preradicals, see \cite{bican1974preradicals, castellini2003categorical, stenstrom1971rings}.

A reflective subcategory is a full subcategory whose inclusion functor admits a left adjoint (called the reflector). Concretely, for each object $X$, the reflector assigns an object $R(X)$ in the subcategory and a morphism $\eta_X \colon X \rightarrow R(X)$ (the unit of the adjunction) such that every morphism from $X$ to an object $A$ in the subcategory factors uniquely through $\eta_X$. This universal property is fundamental; for a detailed treatment and examples, we refer the reader to \cite{borceux1994handbook, maclane1998categories, tholen1987reflective}.

A comonad (or cotriple) on a category $\mathcal{C}$ is an endofunctor $K\colon\mathcal{C}\rightarrow\mathcal{C}$ endowed with natural transformations $\epsilon\colon K\Rightarrow 1_{\mathcal{C}}$ (counit) and $\delta\colon K\Rightarrow K^{2}$ (comultiplication) that satisfy the coassociativity and counit axioms. In particular, every adjunction $F\dashv U\colon\mathcal{D}\rightarrow\mathcal{C}$ canonically induces a comonad $FU$ on $\mathcal{D}$; the comultiplication is given by $F\eta U$, where $\eta$ is the unit of the adjunction, and the counit is simply the adjunction counit $FU\rightarrow 1_{\mathcal{D}}$. We refer the reader to \cite{maclane1998categories} and \cite{borceux1994handbook} for further details.

For comonads $(K,\epsilon,\delta)$ and $(K',\epsilon',\delta')$ on a category $\mathcal{C}$, a morphism of comonads is a natural transformation $\phi\colon K \Rightarrow K'$ such that the following diagrams commute:
\[
\begin{tikzcd}[column sep=4em, row sep=4em]
K \ar[r, "\phi"] \ar[dr, "\epsilon"'] & K' \ar[d, "\epsilon'"] \\
& 1_{\mathcal{C}}
\end{tikzcd}
\qquad
\begin{tikzcd}[column sep=4em, row sep=4em]
K \ar[r, "\phi"] \ar[d, "\delta"] & K' \ar[d, "\delta'"] \\
K^2 \ar[r, "\phi\phi"] & K'^2
\end{tikzcd}
\]
Equivalently, $\phi$ preserves the comultiplication and counit in the sense of \cite{bohm2008monads}; further details on the dual theory for monads can be found in \cite{borceux1994handbook}.

\section{Cofree Functor}
\label{sec:cofree}

For a group $G$, we define a functor $C_G$ from the category of sets to the category of $G$-sets. This functor will be the right adjoint to the forgetful functor $U_G$.

\begin{definition}
Let $G$ be a group. For a set $X$, we define $C_G(X)$ as the set $X^G$ of all functions from $G$ to $X$. We endow $C_G(X)$ with a $G$-action given by
\[
(a\phi)(g) = \phi(a^{-1}g),
\]
for all $\phi \in X^G$ and $a, g \in G$.

For a map $\alpha \colon X \longrightarrow Y$, we define $C_G(\alpha) \colon C_G(X) \longrightarrow C_G(Y)$ by
\[
C_G(\alpha)(\phi)(g) = \alpha(\phi(g)),
\]
for all $\phi \in X^G$ and $g \in G$.
\end{definition}

\begin{proposition}\label{prop:cofree_Gset}
Let $G$ be a group and $X$ a set. Then $C_G(X)$ is a $G$-set.
\end{proposition}

\begin{proof}
Take $a, b, g \in G$ and $\phi \in X^G$. Then
\[
((ab)\phi)(g) = \phi((ab)^{-1}g) = \phi(b^{-1}a^{-1}g) = (b\phi)(a^{-1}g) = (a(b\phi))(g).
\]
Moreover,
\[
(e\phi)(g) = \phi(e^{-1}g) = \phi(eg) = \phi(g).
\]
Therefore $C_G(X)$ is a $G$-set.
\end{proof}

From the preceding proposition, it follows that $C_G$ sends sets to $G$-sets.

\begin{proposition}\label{prop:cofree_functor}
Let $G$ be a group. The assignment $C_G \colon \mathbf{Set} \longrightarrow G\text{-}\mathbf{Set}$ is functorial.
\end{proposition}

\begin{proof}
We first show that $C_G(\alpha)$ is a morphism of $G$-sets for any map $\alpha \colon X \rightarrow Y$. Let $\phi \in X^G$ and $a, g \in G$. Using the definition of the $G$-action, we have
\[
(a C_G(\alpha)(\phi))(g) = C_G(\alpha)(\phi)(a^{-1}g) = \alpha(\phi(a^{-1}g)),
\]
while
\[
C_G(\alpha)(a \phi)(g) = \alpha((a \phi)(g)) = \alpha(\phi(a^{-1}g)).
\]
Hence $(a C_G(\alpha)(\phi)) = C_G(\alpha)(a \phi)$, so $C_G(\alpha)$ is a $G$-equivariant map.

Next, we verify that $C_G$ preserves identities. For a set $X$ and $\phi \in X^G$,
\[
C_G(1_X)(\phi)(g) = 1_X(\phi(g)) = \phi(g),
\]
so $C_G(1_X)(\phi) = \phi$. Thus $C_G(1_X) = 1_{C_G(X)}$.

Finally, we check composition. Let $\alpha \colon X \rightarrow Y$ and $\beta \colon Y \rightarrow Z$ be maps. For $\phi \in X^G$ and $g \in G$,
\[
C_G(\beta\alpha)(\phi)(g) = \beta(\alpha(\phi(g))) = \beta(C_G(\alpha)(\phi)(g)) = C_G(\beta)(C_G(\alpha)(\phi))(g).
\]
Therefore $C_G(\beta\alpha) = C_G(\beta) \circ C_G(\alpha)$, which proves that $C_G$ is a functor.
\end{proof}

Having established that $C_G$ is a functor, we now prove its universal property: it is the right adjoint to the forgetful functor.

\begin{proposition}\label{prop:cofree_adjoint}
Let $G$ be a group. The functor $C_G$ is the right adjoint of the forgetful functor 
$U_G \colon G\text{-}\mathbf{Set} \longrightarrow \mathbf{Set}$, which sends each $G$-set to its underlying set.
\end{proposition}

\begin{proof}
We construct the counit $\epsilon^C \colon U_G C_G \longrightarrow 1_{\mathbf{Set}}$ as follows. 
For a set $X$, define the component 
\[
\epsilon^C_X \colon U_G(C_G(X)) \longrightarrow X
\]
by 
\[
\epsilon^C_X(\phi) = \phi(e),
\]
for $\phi \in X^G$. 
To verify naturality, let $\beta \colon X \rightarrow Y$ be a map of sets. For any $\phi \in X^G$, we have
\[
\epsilon^C_Y\bigl(U_G(C_G(\beta))(\phi)\bigr) 
= U_G(C_G(\beta))(\phi)(e) 
= \beta(\phi(e)) 
= \beta(\epsilon^C_X(\phi)).
\]
Hence the naturality square commutes, so $\epsilon^C$ is a natural transformation.

We now construct the unit $\eta^C \colon 1_{G\text{-}\mathbf{Set}} \longrightarrow C_G U_G$. 
For a $G$-set $X$, define the component 
\[
\eta^C_X \colon X \longrightarrow C_G(U_G(X)) = X^G
\]
by 
\[
\eta^C_X(x)(g) = g^{-1}x,
\]
for $x \in X$ and $g \in G$. 
We first check that $\eta^C_X$ is $G$-equivariant. For $a, g \in G$ and $x \in X$,
\[
(a \eta^C_X(x))(g) = \eta^C_X(x)(a^{-1}g) = (a^{-1}g)^{-1}x = g^{-1}ax,
\]
while
\[
\eta^C_X(ax)(g) = g^{-1}(ax) = g^{-1}ax.
\]
Thus $\eta^C_X(ax) = a \eta^C_X(x)$, so $\eta^C_X$ is a morphism of $G$-sets.

Next, we verify that $\eta^C$ is natural. Let $\alpha \colon X \rightarrow Y$ be a $G$-equivariant map. 
For $x \in X$ and $g \in G$,
\[
C_G(U_G(\alpha))(\eta^C_X(x))(g) 
= \alpha\bigl(\eta^C_X(x)(g)\bigr) 
= \alpha(g^{-1}x) 
= g^{-1}\alpha(x) 
= \eta^C_Y(\alpha(x))(g).
\]
Hence $\eta^C$ is a natural transformation.

It remains to prove the triangle identities.

First, for any $G$-set $X$ and $x \in X$,
\[
\epsilon^C_{U_G(X)}\bigl(U_G(\eta^C_X)(x)\bigr) 
= \eta^C_X(x)(e) 
= e^{-1}x 
= x.
\]
Therefore 
\[
\epsilon^C U_G \circ U_G \eta^C = 1_{U_G}.
\]

Second, for any set $X$ and any $\phi \in C_G(X) = X^G$, we compute the following chain of equalities:
\[
\begin{aligned}
C_G(\epsilon^C_X)\bigl(\eta^C_{C_G(X)}(\phi)\bigr)(g)
  &= \epsilon^C_X\bigl(\eta^C_{C_G(X)}(\phi)(g)\bigr) \\
  &= \epsilon^C_X(g^{-1}\phi) \\
  &= (g^{-1}\phi)(e) \\
  &= \phi((g^{-1})^{-1}e) \\
  &= \phi(g).
\end{aligned}
\]
Thus
\[
C_G(\epsilon^C_X)\bigl(\eta^C_{C_G(X)}(\phi)\bigr) = \phi,
\]
so
\[
C_G \epsilon^C \circ \eta^C C_G = 1_{C_G}.
\]

Since both triangle identities hold, we conclude that $C_G$ is indeed the right adjoint of the forgetful functor $U_G$.
\end{proof}

We can reformulate this adjunction in terms of Hom-functors. 
For a $G$-set $X$ and a set $Y$, we have the natural isomorphism
\[
\theta_{XY} \colon \operatorname{Hom}_{\mathbf{Set}}(U_G(X), Y) 
\longrightarrow \operatorname{Hom}_{G\text{-}\mathbf{Set}}(X, C_G(Y)),
\]
which is defined for any $\alpha \in \operatorname{Hom}_{\mathbf{Set}}(U_G(X), Y)$, 
$x \in X$, and $g \in G$ by
\[
\theta_{XY}(\alpha)(x)(g) = \alpha(g^{-1}x).
\]

\begin{proposition}\label{prop:cofree_faithfulness}
Let $G$ be a group and consider the adjunction $U_G \dashv C_G$ with unit $\eta^C \colon 1_{G\text{-}\mathbf{Set}} \longrightarrow C_G U_G$ and counit $\epsilon^C \colon U_G C_G \longrightarrow 1_{\mathbf{Set}}$. The following statements hold:
\begin{enumerate}
\item The forgetful functor $U_G \colon G\text{-}\mathbf{Set} \longrightarrow \mathbf{Set}$ is faithful.
\item The cofree functor $C_G \colon \mathbf{Set} \longrightarrow G\text{-}\mathbf{Set}$ is faithful.
\item The functor $U_G$ is full if and only if $G$ is the trivial group.
\item The functor $C_G$ is full if and only if $G$ is the trivial group.
\end{enumerate}
\end{proposition}

\begin{proof}
We analyse the components of the unit and the counit.

For a $G$-set $X$, the unit component is given by
\[
\eta^C_X \colon X \longrightarrow X^G, \qquad \eta^C_X(x)(g) = g^{-1}x.
\]
If $\eta^C_X(x_1) = \eta^C_X(x_2)$, then evaluating at $g = e$ yields $x_1 = x_2$. Therefore $\eta^C_X$ is a monomorphism in $\mathbf{Set}$ for every $X$. Since the left adjoint is faithful if and only if all components of the unit are monomorphisms, we conclude that $U_G$ is faithful.

For a set $X$, the counit component is given by
\[
\epsilon^C_X \colon X^G \longrightarrow X, \qquad \epsilon^C_X(\phi) = \phi(e).
\]
For every $x \in X$, the constant function $\phi_x \in X^G$ defined by $\phi_x(g) = x$ satisfies $\epsilon^C_X(\phi_x) = x$. Thus $\epsilon^C_X$ is an epimorphism in $\mathbf{Set}$ for every $X$. Since the right adjoint is faithful if and only if all components of the counit are epimorphisms, it follows that $C_G$ is faithful.

It remains to examine the fullness conditions. The functor $U_G$ is full if and only if every component of the unit $\eta^C_X$ is a split epimorphism. If $G$ is non-trivial, choose a set $X$ with at least two elements, regarded as a trivial $G$-set. In that case, the image of $\eta^C_X$ consists only of the constant functions from $G$ to $X$, since $\eta^C_X(x)(g) = x$ for all $g \in G$. Since there exist non-constant functions in $X^G$ when $|X| \geq 2$, the map $\eta^C_X$ is not surjective, hence not an epimorphism, and therefore cannot be a split epimorphism. Consequently, $U_G$ is not full unless $G$ is trivial. Conversely, if $G$ is the trivial group, then $X^G \cong X$ and $\eta^C_X$ is the identity, so $U_G$ is actually an isomorphism of categories and, in particular, full.

Similarly, the functor $C_G$ is full if and only if every component of the counit $\epsilon^C_X$ is a split monomorphism. If $G$ is non-trivial and $X$ has at least two elements, the map $\epsilon^C_X$ is not injective, since two different functions $G \rightarrow X$ can agree at $e$. Hence $\epsilon^C_X$ is not a monomorphism, and thus not a split monomorphism. It follows that $C_G$ is not full for non-trivial $G$. If $G$ is trivial, then $\epsilon^C_X$ is again the identity, and $C_G$ is full. Therefore both adjoint functors are full exactly in the trivial case.
\end{proof}

\section{Burnside's Lemma}
\label{sec:burnside}

The following result characterizes the fixed points of the cofree $G$-set $C_G(X)$, which will be essential for counting orbits via Burnside's Lemma.

\begin{proposition}\label{prop:fixed_points_global}
Let $G$ be a group and $X$ a set. Then $\phi \in C_G(X)^G$ if and only if there exists $x \in X$ such that $\phi(g) = x$ for all $g \in G$.
\end{proposition}

\begin{proof}
$(\Rightarrow)$ Suppose $\phi \in C_G(X)^G$. Set $x = \phi(e)$. For any $g \in G$, since $\phi$ is fixed by the $G$-action, we have $g^{-1}\phi = \phi$. Evaluating this equality at $e \in G$ gives
\[
(g^{-1}\phi)(e) = \phi(e).
\]
But by the definition of the action on $C_G(X)$,
\[
(g^{-1}\phi)(e) = \phi((g^{-1})^{-1}e) = \phi(g).
\]
Hence $\phi(g) = \phi(e) = x$ for all $g \in G$. Therefore $\phi$ is constant.

$(\Leftarrow)$ Conversely, suppose there is $x \in X$ such that $\phi(g) = x$ for all $g \in G$. For any $a \in G$ and any $g \in G$, we have
\[
(a \phi)(g) = \phi(a^{-1}g) = x = \phi(g).
\]
Thus $a \phi = \phi$ for every $a \in G$, so $\phi \in C_G(X)^G$.
\end{proof}

We now specialize the previous characterization to the fixed points of a single group element $g \in G$. This will be particularly useful when we apply Burnside's Lemma, as it requires counting the elements fixed by each individual group element.

\begin{proposition}\label{prop:fixed_points_element}
Let $G$ be a group, $g \in G$, and $X$ a set. Then $\phi \in C_G(X)^g$ if and only if, for each coset $Ha$ with $a \in G$ and $H = \langle g \rangle$, the restriction $\phi|_{Ha}$ is constant.
\end{proposition}

\begin{proof}
$(\Rightarrow)$ Suppose $\phi \in C_G(X)^g$, i.e. $g \phi = \phi$. 
For any $a \in G$, let $x = \phi(a)$. Since $g \phi = \phi$, we also have $g^{-1} \phi = \phi$. Evaluating at $a$ gives
\[
(g^{-1}\phi)(a) = \phi(a) = x.
\]
But by definition of the action on $C_G(X)$,
\[
(g^{-1}\phi)(a) = \phi((g^{-1})^{-1}a) = \phi(g a).
\]
Hence $\phi(g a) = \phi(a) = x$. By induction, $\phi(g^n a) = x$ for all $n \ge 0$. 
For negative powers, since $g \phi = \phi$, we have $(g \phi)(a) = \phi(a)$, which gives
\[
(g \phi)(a) = \phi(g^{-1}a) = \phi(a) = x.
\]
Thus, again by induction, $\phi(g^{-n} a) = x$ for all $n \ge 0$. Therefore $\phi$ takes the constant value $x$ on every element of the coset $Ha = \{g^n a \mid n \in \mathbb{Z}\}$. Hence $\phi|_{Ha}$ is constant.

$(\Leftarrow)$ Conversely, suppose that for each coset $Ha$ with $H = \langle g \rangle$, the restriction of $\phi$ to $Ha$ is constant. 
Let $h \in G$. Then $h$ belongs to exactly one coset, say $h \in Ha$, so there exists $n \in \mathbb{Z}$ with $h = g^n a$. Since both $g^{-1}h = g^{n-1}a$ and $h = g^n a$ lie in the same coset $Ha$, and $\phi$ is constant on this coset, we have
\[
\phi(g^{-1} h) = \phi(h).
\]
Therefore,
\[
(g \phi)(h) = \phi(g^{-1} h) = \phi(h).
\]
Since $h \in G$ was arbitrary, it follows that $g \phi = \phi$, and thus $\phi \in C_G(X)^g$.
\end{proof}

The characterization of fixed points obtained above is particularly useful in conjunction with the following classical counting result, which relates the number of orbits of a finite group action to the average number of elements fixed by each group element.

\begin{proposition}[Burnside's Lemma]\label{prop:burnside}
Let $G$ be a finite group and $X$ a finite $G$-set. Then
\[
\lvert X/G \rvert = \frac{1}{\lvert G \rvert} \sum_{g \in G} \lvert X^g \rvert,
\]
where $X^g = \{ x \in X \mid g x = x \}$ denotes the set of fixed points of $g$.
\end{proposition}
\begin{proof}
This is the classical Burnside's Lemma; see for instance \cite{rotman2012introduction}.
\end{proof}

We now apply Burnside's Lemma to the cofree $G$-set $C_G(X)$. This yields a classical number-theoretic result (Fermat's little theorem) as a direct corollary.

\begin{proposition}\label{prop:fermat_prime}
Let $n$ be a natural number and $p$ a prime number. Then $p \mid n^p + (p-1)n$.
\end{proposition}

\begin{proof}
Let $X$ be a set with $n$ elements and consider the cofree $G$-set $C_{\mathbb{Z}_p}(X)$. 
Since $p$ is prime, every non-zero element $x \in \mathbb{Z}_p$ generates the whole group. 
Thus, for any such $x$, the set of fixed points of $x$ equals the set of fixed points of the entire group:
\[
C_{\mathbb{Z}_p}(X)^x = C_{\mathbb{Z}_p}(X)^{\mathbb{Z}_p}.
\]
By our previous characterization of the fixed points of the cofree $G$-set, the elements of this set are precisely the constant maps from $\mathbb{Z}_p$ to $X$. Hence 
\[
\left\lvert C_{\mathbb{Z}_p}(X)^x \right\rvert = n
\]
for every $x \neq 0$.

On the other hand, for the identity element $0 \in \mathbb{Z}_p$, we have
\[
C_{\mathbb{Z}_p}(X)^0 = C_{\mathbb{Z}_p}(X) = X^{\mathbb{Z}_p},
\]
so its cardinality is
\[
\left\lvert C_{\mathbb{Z}_p}(X)^0 \right\rvert = n^p.
\]

Applying Burnside's Lemma to the action of $\mathbb{Z}_p$ on $C_{\mathbb{Z}_p}(X)$, we obtain
\[
\left\lvert C_{\mathbb{Z}_p}(X)/\mathbb{Z}_p \right\rvert
= \frac{1}{p} \left( n^p + (p-1)n \right).
\]
Since the left-hand side is a non-negative integer, the right-hand side must also be an integer. Therefore
\[
p \mid n^p + (p-1)n,
\]
as required.
\end{proof}

The previous result for a prime $p$ generalizes naturally to prime powers $p^m$. The following proposition provides the corresponding congruence modulo $p^m$, obtained by applying Burnside's Lemma to the cyclic group $\mathbb{Z}_{p^m}$.

\begin{proposition}\label{prop:fermat_prime_power}
Let $n$ and $m$ be natural numbers, and let $p$ be a prime number. Then
\[
p^m \mid n^{p^m} + \sum_{s=1}^m (p^s - p^{s-1}) n^{p^{m-s}}.
\]
\end{proposition}

\begin{proof}
Let $X$ be a set with $n$ elements, and consider the cofree $G$-set $C_{\mathbb{Z}_{p^m}}(X)$.

Recall that in the cyclic group $\mathbb{Z}_{p^m}$, for each $s = 0, 1, \dots, m$, there is a unique subgroup of order $p^s$. Consequently, the number of elements of order exactly $p^s$ is
\[
p^s - p^{s-1}
\]
for $s = 1, \dots, m$ (the identity element has order $p^0 = 1$).

Now take an element $x \in \mathbb{Z}_{p^m}$ of order $p^s$, where $1 \le s \le m$. The subgroup generated by $x$ is $\langle x \rangle$, whose order is $p^s$. Hence its index in $\mathbb{Z}_{p^m}$ is
\[
[\mathbb{Z}_{p^m} : \langle x \rangle] = \frac{p^m}{p^s} = p^{m-s}.
\]
By the characterization of fixed points in the cofree $G$-set, $\phi \in C_{\mathbb{Z}_{p^m}}(X)^x$ if and only if $\phi$ is constant on each coset of $\langle x \rangle$. Since there are exactly $p^{m-s}$ such cosets, the number of such maps is
\[
\left\lvert C_{\mathbb{Z}_{p^m}}(X)^x \right\rvert = n^{p^{m-s}}.
\]

For the identity element $0 \in \mathbb{Z}_{p^m}$, we have $C_{\mathbb{Z}_{p^m}}(X)^0 = C_{\mathbb{Z}_{p^m}}(X) = X^{\mathbb{Z}_{p^m}}$, so
\[
\left\lvert C_{\mathbb{Z}_{p^m}}(X)^0 \right\rvert = n^{p^m}.
\]

Applying Burnside's Lemma to the action of $\mathbb{Z}_{p^m}$ on $C_{\mathbb{Z}_{p^m}}(X)$, we obtain
\[
\left\lvert C_{\mathbb{Z}_{p^m}}(X)/\mathbb{Z}_{p^m} \right\rvert
= \frac{1}{p^m} \left( n^{p^m} + \sum_{s=1}^m (p^s - p^{s-1}) n^{p^{m-s}} \right).
\]
Since the left-hand side is a non-negative integer, the right-hand side must also be an integer. Therefore
\[
p^m \mid n^{p^m} + \sum_{s=1}^m (p^s - p^{s-1}) n^{p^{m-s}},
\]
as required.
\end{proof}

We now extend the previous prime-power case to an arbitrary cyclic group of order $n$. The following classical congruence, which generalizes Fermat's little theorem, emerges naturally from the same cofree construction.

\begin{proposition}\label{prop:fermat_general}
Let $n$ and $k$ be natural numbers. Then
\[
n \mid \sum_{d \mid n} \varphi(d)\, k^{\,n/d},
\]
where $\varphi$ denotes Euler's totient function.
\end{proposition}

\begin{proof}
Let $G = \mathbb{Z}_n$ be the cyclic group of order $n$, and let $X$ be a set with $|X| = k$. We consider the cofree $G$-set $C_G(X) = X^G$.

By the characterization of fixed points obtained in Proposition~\ref{prop:fixed_points_element}, for an element $g \in G$, the set of fixed points $C_G(X)^g$ consists precisely of those functions that are constant on each coset of the subgroup $\langle g \rangle$. Since there are $[G : \langle g \rangle]$ such cosets, and each coset can be assigned independently an element of $X$, we have
\[
\left\lvert C_G(X)^g \right\rvert = |X|^{[G : \langle g \rangle]} = k^{\, n / \operatorname{ord}(g)}.
\]

Now, in the cyclic group $\mathbb{Z}_n$, for each divisor $d$ of $n$, there exist exactly $\varphi(d)$ elements of order $d$. Grouping the elements of $G$ by their order and applying Burnside's Lemma (Proposition~\ref{prop:burnside}) to the $G$-set $C_G(X)$, we obtain
\[
\left\lvert C_G(X) / G \right\rvert
= \frac{1}{n} \sum_{g \in G} \left\lvert C_G(X)^g \right\rvert
= \frac{1}{n} \sum_{d \mid n} \varphi(d)\, k^{\, n/d}.
\]

Since the left-hand side represents the number of orbits of the $G$-action on $C_G(X)$, it is necessarily an integer. Hence,
\[
n \mid \sum_{d \mid n} \varphi(d)\, k^{\, n/d}.
\]
Thus the result follows.
\end{proof}

We now consider the elementary abelian $p$-group $(\mathbb{Z}_p)^r$ instead of the cyclic $p$-group. This choice yields a different classical congruence, which is a slight variant of the previous results.

\begin{proposition}\label{prop:fermat_elementary_abelian}
Let $p$ be a prime number, $r$ a natural number, and $k$ a natural number. Then
\[
p^r \mid k^{\,p^r} + (p^r - 1)\, k^{\,p^{r-1}}.
\]
\end{proposition}

\begin{proof}
Let $G = (\mathbb{Z}_p)^r$ be the elementary abelian $p$-group of order $p^r$, and let $X$ be a set with $|X| = k$. We consider the cofree $G$-set $C_G(X) = X^G$.

We apply Burnside's Lemma (Proposition~\ref{prop:burnside}) to this $G$-set. To do so, we compute the number of fixed points $|C_G(X)^g|$ for each $g \in G$.

By the characterization of fixed points obtained in Proposition~\ref{prop:fixed_points_element}, for any $g \in G$, we have
\[
\left\lvert C_G(X)^g \right\rvert = |X|^{[G : \langle g \rangle]} = k^{\, p^r / \operatorname{ord}(g)}.
\]

Now we classify the elements of $G$ by their order. For the identity element $g=e$, we have $\operatorname{ord}(g)=1$, and therefore $\left\lvert C_G(X)^e \right\rvert = |C_G(X)| = k^{\,p^r}$. For any non-identity element $g \neq e$, since $G$ is an elementary abelian $p$-group, it follows that $\operatorname{ord}(g)=p$. Hence, $\left\lvert C_G(X)^g \right\rvert = k^{\, p^r / p} = k^{\,p^{r-1}}$, and the number of such elements in $G$ is exactly $p^r - 1$.

Summing the contributions of all elements of $G$ and applying Burnside's Lemma, we obtain that the number of orbits of the $G$-action on $C_G(X)$ is
\[
\left\lvert C_G(X) / G \right\rvert
= \frac{1}{p^r}
\left( k^{\,p^r} + (p^r - 1)\, k^{\,p^{r-1}} \right).
\]

Since the number of orbits is necessarily an integer, the numerator must be divisible by $p^r$. Therefore,
\[
p^r \mid k^{\,p^r} + (p^r - 1)\, k^{\,p^{r-1}}.
\]
This completes the proof.
\end{proof}

We now turn to a non-abelian example. Applying the cofree construction to the dihedral group yields another classical congruence, which further illustrates the versatility of Burnside's Lemma in this context.

Consider the dihedral group $D_n$ of order $2n$, defined for $n \geq 3$ as the group of symmetries of a regular $n$-gon. Algebraically, it admits the presentation
\[
D_n = \langle r, s \mid r^n = 1,\; s^2 = 1,\; srs = r^{-1} \rangle,
\]
where $r$ denotes a rotation of angle $2\pi/n$ and $s$ denotes a reflection. Thus every element of $D_n$ can be uniquely written as $r^i$ or $sr^i$, with $0 \leq i < n$. The rotations $r^i$ form a cyclic subgroup of order $n$, while the remaining $n$ elements are reflections, each of order $2$. Hence $D_n$ is non-abelian for $n \geq 3$, since $sr \neq rs$ in general. This structure makes it an excellent test case for the cofree construction: the different orders of its elements lead to distinct contributions in the fixed-point sum $\sum_{g \in D_n} |C_{D_n}(X)^g|$, which will ultimately yield a classical congruence depending on the parity of $n$.

\begin{proposition}\label{prop:dihedral_congruence}
Let $n \ge 3$ be a natural number and $k$ a natural number. Let $D_n$ be the dihedral group of order $2n$. Then
\[
2n \mid \sum_{d \mid n} \varphi(d)\, k^{\,2n/d} + n k^{\,n},
\]
where $\varphi$ denotes Euler's totient function.
\end{proposition}

\begin{proof}
Let $G = D_n$, the dihedral group of order $2n$, which consists of $n$ rotations and $n$ reflections. Let $X$ be a set with $|X| = k$, and consider the cofree $G$-set $C_G(X) = X^G$.

We apply Burnside's Lemma (Proposition~\ref{prop:burnside}) to this $G$-set. To do so, we compute the number of fixed points $|C_G(X)^g|$ for each $g \in G$.

By the characterization of fixed points obtained in Proposition~\ref{prop:fixed_points_element}, for any element $g \in G$, we have
\[
\left\lvert C_G(X)^g \right\rvert = |X|^{[G : \langle g \rangle]} = k^{\, 2n / \operatorname{ord}(g)}.
\]

We now classify the elements of $D_n$ by their order.

\begin{itemize}
\item \textbf{Rotations:} There are $n$ rotations. For each divisor $d$ of $n$, the cyclic subgroup of rotations contains exactly $\varphi(d)$ elements of order $d$. For such an element $g$, we have $\operatorname{ord}(g) = d$, hence
\[
\left\lvert C_G(X)^g \right\rvert = k^{\, 2n / d}.
\]

\item \textbf{Reflections:} There are $n$ reflections. Every reflection has order $2$. Therefore, for each reflection $g$, we have $\operatorname{ord}(g) = 2$, and consequently
\[
\left\lvert C_G(X)^g \right\rvert = k^{\, 2n / 2} = k^{\, n}.
\]
Since there are $n$ reflections, their total contribution to the sum of fixed points is $n k^{\, n}$.
\end{itemize}

Summing the contributions from all elements of $G$ and applying Burnside's Lemma, we obtain that the number of orbits of the $G$-action on $C_G(X)$ is
\[
\left\lvert C_G(X) / G \right\rvert
= \frac{1}{2n}
\left( \sum_{d \mid n} \varphi(d)\, k^{\, 2n/d} + n k^{\, n} \right).
\]

Since the number of orbits is necessarily an integer, the numerator must be divisible by $2n$. Therefore,
\[
2n \mid \sum_{d \mid n} \varphi(d)\, k^{\, 2n/d} + n k^{\, n}.
\]
This completes the proof.
\end{proof}

\section{Approximations to the Cofree Functor}
\label{sec:approx}

We now introduce a family of approximations to the cofree functor, parameterized by subgroups of $G$. These generalized cofree $G$-sets will be useful when studying induced actions.

\begin{definition}\label{def:cofree_induced}
Let $G$ be a group and $H$ a subgroup of $G$. For a set $X$, we define $C_H^G(X)$ as the set of all functions from the coset space $G/H$ to $X$.

We endow $C_H^G(X)$ with a $G$-action defined by
\[
(g \phi)(aH) = \phi(g^{-1}aH)
\]
for all $\phi \in C_H^G(X)$, $aH \in G/H$, and $g \in G$. This action is well-defined and makes $C_H^G(X)$ a $G$-set.

For a map $f \colon X \longrightarrow Y$, we define the induced map
\[
C_H^G(f) \colon C_H^G(X) \longrightarrow C_H^G(Y)
\]
by
\[
C_H^G(f)(\phi) = f \circ \phi,
\]
for all $\phi \in C_H^G(X)$.
\end{definition}

When $H = \{e\}$ is the trivial subgroup, we have $G/H \cong G$, so the definition above coincides with the original cofree functor $C_{\{e\}}^G(X) \cong C_G(X)$. Thus the original cofree functor is a special case of this more general construction.

Having defined the generalized cofree $G$-set $C_H^G(X)$, we now verify that this construction is functorial in the set variable $X$.

\begin{proposition}\label{prop:cofree_induced_functor}
Let $G$ be a group and $H$ a subgroup of $G$. Then the assignment
\[
C_H^G \colon \mathbf{Set} \longrightarrow G\text{-}\mathbf{Set}
\]
is functorial.
\end{proposition}

\begin{proof}
First, let $f \colon X \longrightarrow Y$ be a map of sets. We verify that $C_H^G(f) \colon C_H^G(X) \longrightarrow C_H^G(Y)$ is a morphism of $G$-sets. For any $\phi \in C_H^G(X)$, $aH \in G/H$, and $g \in G$, we have
\[
\bigl(g C_H^G(f)(\phi)\bigr)(aH)
= C_H^G(f)(\phi)(g^{-1}aH)
= f\bigl(\phi(g^{-1}aH)\bigr),
\]
while
\[
C_H^G(f)(g \phi)(aH)
= f\bigl((g \phi)(aH)\bigr)
= f\bigl(\phi(g^{-1}aH)\bigr).
\]
Hence $g C_H^G(f)(\phi) = C_H^G(f)(g \phi)$, so $C_H^G(f)$ is $G$-equivariant.

Next, for the identity map $1_X \colon X \longrightarrow X$, and for any $\phi \in C_H^G(X)$, we have
\[
C_H^G(1_X)(\phi) = 1_X \circ \phi = \phi.
\]
Thus $C_H^G(1_X) = 1_{C_H^G(X)}$.

Finally, let $f \colon X \longrightarrow Y$ and $g \colon Y \longrightarrow Z$ be maps of sets. For any $\phi \in C_H^G(X)$,
\[
C_H^G(g \circ f)(\phi) = (g \circ f) \circ \phi = g \circ (f \circ \phi) = C_H^G(g)\bigl(C_H^G(f)(\phi)\bigr).
\]
Therefore $C_H^G(g \circ f) = C_H^G(g) \circ C_H^G(f)$.

Since all three conditions are satisfied, $C_H^G$ is a functor.
\end{proof}

The functors $C_G$ and $C_H^G$ are closely related, but $C_H^G$ cannot be a right adjoint to the full forgetful functor $U_G \colon G\text{-}\mathbf{Set} \rightarrow \mathbf{Set}$. 
Besides the uniqueness of adjoints (since $C_G$ is already a right adjoint to $U_G$), a direct attempt to construct the natural isomorphism reveals a fundamental well-definedness issue.

Indeed, suppose we try to build a natural isomorphism
\[
\theta^H_{XY} \colon \operatorname{Hom}_{\mathbf{Set}}(U_G(X), Y) 
\longrightarrow \operatorname{Hom}_{G\text{-}\mathbf{Set}}(X, C_H^G(Y))
\]
for a $G$-set $X$ and a set $Y$, mimicking the cofree case. A natural candidate would be
\[
\theta^H_{XY}(\alpha)(x)(gH) = \alpha(g^{-1}x),
\]
for $\alpha \in \operatorname{Hom}_{\mathbf{Set}}(U_G(X), Y)$, $x \in X$, and $gH \in G/H$. However, this map is not well-defined in general. If $gH = g'hH$ for some $h \in H$ (i.e., $g' = gh$), we would need
\[
\alpha(g^{-1}x) = \alpha((gh)^{-1}x) = \alpha(h^{-1}g^{-1}x).
\]
For this equality to hold for every arbitrary map $\alpha \colon X \rightarrow Y$, we must have $g^{-1}x = h^{-1}g^{-1}x$ for all $x \in X$, which is equivalent to $h x = x$ for all $h \in H$. In other words, $H$ must act trivially on $X$. Since this is not true for arbitrary $G$-sets, the adjunction fails.

Motivated by this obstruction, we introduce a restricted class of $G$-sets that arise naturally from the transitive $G$-set $G/H$.

\begin{definition}\label{def:H_generated}
Let $G$ be a group and $H$ a subgroup of $G$. We say that a $G$-set $X$ is \emph{$H$-generated} if there exists a $G$-equivariant epimorphism
\[
\coprod_{i \in I} G/H \longrightarrow X
\]
for some indexing set $I$. Equivalently, $X$ is a quotient of a coproduct of copies of the transitive $G$-set $G/H$. We denote the full subcategory of $H$-generated $G$-sets by $G\text{-}\mathbf{Set}_H$.
\end{definition}

We now characterize the $H$-generated $G$-sets among those that are disjoint unions of transitive $G$-sets. The condition is naturally expressed in terms of conjugacy of subgroups.

\begin{proposition}\label{prop:H_generated_criterion}
Let $G$ be a group and let
\[
X = \bigsqcup_{i \in I} G/H_i
\]
be a disjoint union of transitive $G$-sets. Then $X$ is $H$-generated if and only if for each $i \in I$, the subgroup $H$ is subconjugate to $H_i$ in $G$, that is, there exists an element $g_i \in G$ such that
\[
g_i^{-1} H g_i \subseteq H_i.
\]
\end{proposition}

\begin{proof}
$(\Rightarrow)$ Suppose $X$ is $H$-generated. Then there exists a $G$-equivariant epimorphism
\[
\pi \colon \coprod_{j \in J} G/H \longrightarrow X,
\]
for some indexing set $J$. Fix an index $i \in I$. Since $G/H_i$ is a connected component of $X$, and $\pi$ is surjective, there must be some index $j \in J$ such that the image of the corresponding copy of $G/H$ under $\pi$ is exactly $G/H_i$ (because the image of a transitive $G$-set under a $G$-equivariant map is a single orbit). Thus we obtain a $G$-equivariant epimorphism
\[
\pi_{ij} \colon G/H \longrightarrow G/H_i.
\]

Now, any $G$-equivariant map $f \colon G/H \rightarrow G/H_i$ is determined by the image of the coset $H$. Let $f(H) = g H_i$ for some $g \in G$. By equivariance, for any $h \in H$,
\[
g H_i = f(H) = f(hH) = h f(H) = h g H_i.
\]
Hence $h g H_i = g H_i$, which implies $g^{-1} h g \in H_i$ for all $h \in H$. Therefore $g^{-1} H g \subseteq H_i$, so $H$ is subconjugate to $H_i$. Since $i$ was arbitrary, the condition holds for all $i$.

$(\Leftarrow)$ Conversely, suppose that for each $i \in I$ there exists $g_i \in G$ with $g_i^{-1} H g_i \subseteq H_i$. Then the map
\[
\pi_i \colon G/H \longrightarrow G/H_i, \qquad aH \longmapsto a g_i H_i
\]
is well-defined and $G$-equivariant (because $g_i^{-1} H g_i \subseteq H_i$). Moreover, it is surjective since $G$ acts transitively on $G/H_i$.

Taking the coproduct of all these maps over $i \in I$, we obtain a $G$-equivariant map
\[
\pi \colon \coprod_{i \in I} G/H \longrightarrow \bigsqcup_{i \in I} G/H_i = X,
\]
which is surjective because each component map is. Hence $X$ is $H$-generated.
\end{proof}

The obstruction to the adjunction of $C_H^G$ motivates the study of a related concept: the $H$-trace of a $G$-set. This will allow us to isolate the part of a $G$-set that is "visible" from the transitive $G$-set $G/H$.

\begin{definition}\label{def:H_trace}
Let $G$ be a group and $H$ a subgroup of $G$. For a $G$-set $X$, we define its \emph{$H$-trace} as
\[
\operatorname{Tr}_H(X) = \bigcup \bigl\{ \operatorname{im}(\alpha) \mid \alpha \in \operatorname{Hom}_{G\text{-}\mathbf{Set}}(G/H, X) \bigr\}.
\]
In other words, $\operatorname{Tr}_H(X)$ is the union of the images of all $G$-equivariant maps from the transitive $G$-set $G/H$ into $X$.
\end{definition}

\begin{proposition}\label{prop:tr_preradical}
Let $G$ be a group and $H$ a subgroup of $G$. Then $\operatorname{Tr}_H$ is a preradical on the category $G\text{-}\mathbf{Set}$.
\end{proposition}

\begin{proof}
We must verify the two defining properties of a preradical.

First, we show that $\operatorname{Tr}_H(X)$ is a $G$-subobject of $X$. Let $X$ be a $G$-set, take $x \in \operatorname{Tr}_H(X)$, and let $g \in G$. By definition of the trace, there exists a $G$-equivariant map $\alpha \colon G/H \longrightarrow X$ and an element $aH \in G/H$ such that $\alpha(aH) = x$. Since $\alpha$ is $G$-equivariant,
\[
g x = g \alpha(aH) = \alpha(g aH).
\]
Thus $g x \in \operatorname{im}(\alpha) \subseteq \operatorname{Tr}_H(X)$. Therefore $\operatorname{Tr}_H(X)$ is a $G$-subset of $X$.

Second, we verify functoriality under morphisms. Let $\beta \colon X \longrightarrow Y$ be a $G$-equivariant map, and let $x \in \operatorname{Tr}_H(X)$. Again, there exists a $G$-equivariant map $\alpha \colon G/H \longrightarrow X$ such that $x \in \operatorname{im}(\alpha)$. Then the composition
\[
\beta \circ \alpha \colon G/H \longrightarrow Y
\]
is also $G$-equivariant, and
\[
\beta(x) = \beta(\alpha(aH)) = (\beta \circ \alpha)(aH) \in \operatorname{im}(\beta \circ \alpha) \subseteq \operatorname{Tr}_H(Y).
\]
Hence $\beta(\operatorname{Tr}_H(X)) \subseteq \operatorname{Tr}_H(Y)$. This proves that $\operatorname{Tr}_H$ is a preradical on $G\text{-}\mathbf{Set}$.
\end{proof}

Observe that $\operatorname{Tr}_H(X)$ is precisely the largest $H$-generated $G$-subobject of $X$. Indeed, if $X$ is itself $H$-generated, then $\operatorname{Tr}_H(X) = X$.

Having established that $\operatorname{Tr}_H$ is a preradical, we now show that it is idempotent. This means that applying the trace operation twice yields no new elements; the $H$-trace of an $H$-trace is already the full $H$-trace.

\begin{proposition}\label{prop:tr_idempotent}
Let $G$ be a group and $H$ a subgroup of $G$. Then the preradical $\operatorname{Tr}_H$ is idempotent. That is, for every $G$-set $X$,
\[
\operatorname{Tr}_H\bigl(\operatorname{Tr}_H(X)\bigr) = \operatorname{Tr}_H(X).
\]
\end{proposition}

\begin{proof}
Let $X$ be a $G$-set and set $Y = \operatorname{Tr}_H(X)$. We must show that $\operatorname{Tr}_H(Y) = Y$.

First, we prove the inclusion $Y \subseteq \operatorname{Tr}_H(Y)$. Take $x \in Y$. By the definition of the $H$-trace, there exists a $G$-equivariant map $\alpha \colon G/H \longrightarrow X$ and an element $aH \in G/H$ such that $\alpha(aH) = x$. Since $\operatorname{im}(\alpha) \subseteq \operatorname{Tr}_H(X) = Y$, we may consider the corestriction
\[
\overline{\alpha} \colon G/H \longrightarrow Y, \qquad \overline{\alpha}(gH) = \alpha(gH).
\]
This map is well-defined and remains $G$-equivariant because it is the same formula as $\alpha$ with its codomain restricted to $Y$. Moreover,
\[
\overline{\alpha}(aH) = \alpha(aH) = x.
\]
Thus $x \in \operatorname{im}(\overline{\alpha}) \subseteq \operatorname{Tr}_H(Y)$. Hence $Y \subseteq \operatorname{Tr}_H(Y)$.

Conversely, since $\operatorname{Tr}_H$ is a preradical, we already know from Proposition~\ref{prop:tr_preradical} that $\operatorname{Tr}_H(Y)$ is a $G$-subset of $Y$. Therefore $\operatorname{Tr}_H(Y) \subseteq Y$.

Combining both inclusions, we obtain
\[
\operatorname{Tr}_H\bigl(\operatorname{Tr}_H(X)\bigr) = \operatorname{Tr}_H(Y) = Y = \operatorname{Tr}_H(X).
\]
Thus $\operatorname{Tr}_H$ is idempotent.
\end{proof}

The following result provides a fundamental criterion: a $G$-set is $H$-generated precisely when its $H$-trace is the whole set.

\begin{proposition}\label{prop:H_generated_trace_criterion}
Let $G$ be a group, $H$ a subgroup of $G$, and $X$ a $G$-set. Then $X$ is $H$-generated if and only if $\operatorname{Tr}_H(X) = X$.
\end{proposition}

\begin{proof}
$(\Rightarrow)$ Suppose $X$ is $H$-generated. Then there exists a set $I$ and a $G$-equivariant epimorphism
\[
\pi \colon \coprod_{i \in I} G/H \longrightarrow X.
\]
Let $x \in X$ be arbitrary. Since $\pi$ is surjective, there exist $i \in I$ and $gH \in G/H$ such that
\[
\pi(\iota_i(gH)) = x,
\]
where $\iota_i \colon G/H \longrightarrow \coprod_{i \in I} G/H$ is the $i$-th canonical inclusion of the coproduct.
Now, the composition $\pi \circ \iota_i \colon G/H \longrightarrow X$ is a $G$-equivariant map, and its image contains $x$. Hence, by the definition of the $H$-trace,
\[
x \in \operatorname{im}(\pi \circ \iota_i) \subseteq \operatorname{Tr}_H(X).
\]
Since $x \in X$ was arbitrary, we have $X \subseteq \operatorname{Tr}_H(X)$. The reverse inclusion $\operatorname{Tr}_H(X) \subseteq X$ holds trivially by the definition of the trace. Therefore $\operatorname{Tr}_H(X) = X$.

$(\Leftarrow)$ Conversely, suppose $\operatorname{Tr}_H(X) = X$. For each $x \in X$, choose a $G$-equivariant map
\[
\alpha_x \colon G/H \longrightarrow X
\]
such that $x \in \operatorname{im}(\alpha_x)$; such a map exists precisely because $x \in \operatorname{Tr}_H(X)$.

By the universal property of the coproduct, the family $\{\alpha_x\}_{x \in X}$ induces a unique $G$-equivariant map
\[
\alpha \colon \coprod_{x \in X} G/H \longrightarrow X.
\]
We claim that $\alpha$ is surjective. Indeed, for any $x \in X$, there exists some $gH \in G/H$ such that $\alpha_x(gH) = x$. Therefore,
\[
\alpha(\iota_x(gH)) = \alpha_x(gH) = x,
\]
so $x$ lies in the image of $\alpha$. Hence $\alpha$ is a $G$-equivariant epimorphism. This proves that $X$ is $H$-generated.
\end{proof}

We now show that the subcategory of $H$-generated $G$-sets is a coreflective subcategory of the category of all $G$-sets, with the $H$-trace functor serving as the coreflector.

\begin{proposition}\label{prop:tr_coreflector}
Let $G$ be a group and $H$ a subgroup of $G$. Then the full subcategory $G\text{-}\mathbf{Set}_H$ of $H$-generated $G$-sets is coreflective in $G\text{-}\mathbf{Set}$. The coreflector is the $H$-trace functor $\operatorname{Tr}_H$.
\end{proposition}

\begin{proof}
We must verify the universal property of the coreflection. For every $G$-set $Y$, we need to show that the inclusion
\[
\iota_{\operatorname{Tr}_H(Y)} \colon \operatorname{Tr}_H(Y) \hookrightarrow Y
\]
has the following property: for every $H$-generated $G$-set $X$ and every $G$-equivariant map $\alpha \colon X \longrightarrow Y$, there exists a unique $G$-equivariant map
\[
\overline{\alpha} \colon X \longrightarrow \operatorname{Tr}_H(Y)
\]
such that $\alpha = \iota_{\operatorname{Tr}_H(Y)} \circ \overline{\alpha}$.

Since $X$ is $H$-generated, by Proposition~\ref{prop:H_generated_trace_criterion} we have $\operatorname{Tr}_H(X) = X$. Moreover, since $\operatorname{Tr}_H$ is a preradical (Proposition~\ref{prop:tr_preradical}), the map $\alpha$ restricts to a $G$-equivariant map
\[
\operatorname{Tr}_H(\alpha) \colon \operatorname{Tr}_H(X) \longrightarrow \operatorname{Tr}_H(Y).
\]
But $\operatorname{Tr}_H(X) = X$, so we obtain a $G$-equivariant map
\[
\overline{\alpha} := \operatorname{Tr}_H(\alpha) \colon X \longrightarrow \operatorname{Tr}_H(Y).
\]

We claim that this map satisfies the desired factorization. Indeed, for any $x \in X$,
\[
\iota_{\operatorname{Tr}_H(Y)}(\overline{\alpha}(x))
= \iota_{\operatorname{Tr}_H(Y)}(\operatorname{Tr}_H(\alpha)(x))
= \alpha(x),
\]
where the last equality follows from the definition of $\operatorname{Tr}_H(\alpha)$ as the restriction of $\alpha$ to the trace. Hence $\alpha = \iota_{\operatorname{Tr}_H(Y)} \circ \overline{\alpha}$.

It remains to prove uniqueness. Suppose there is another $G$-equivariant map
\[
\beta \colon X \longrightarrow \operatorname{Tr}_H(Y)
\]
such that $\alpha = \iota_{\operatorname{Tr}_H(Y)} \circ \beta$. Since $\iota_{\operatorname{Tr}_H(Y)}$ is a monomorphism (being an inclusion of a subset), we can cancel it on the left:
\[
\iota_{\operatorname{Tr}_H(Y)} \circ \overline{\alpha}
= \iota_{\operatorname{Tr}_H(Y)} \circ \beta
\quad \Longrightarrow \quad
\overline{\alpha} = \beta.
\]
Thus the factorization is unique.

Therefore, for every $Y$, the inclusion $\operatorname{Tr}_H(Y) \hookrightarrow Y$ is a coreflection morphism. This proves that $G\text{-}\mathbf{Set}_H$ is a coreflective subcategory of $G\text{-}\mathbf{Set}$, with coreflector $\operatorname{Tr}_H$.
\end{proof}

For normal subgroups, the set of cosets $G/H$ admits a well-defined right action of $G$, which will be useful when studying the approximations to the cofree functor and their restrictions.

\begin{proposition}\label{prop:normal_right_action}
Let $G$ be a group and $H$ a normal subgroup of $G$. If $gH = hH$, then
\[
gaH = haH \quad \text{for all } a \in G.
\]
In other words, right multiplication by any element of $G$ is well-defined on the set of left cosets $G/H$.
\end{proposition}

\begin{proof}
Assume $gH = hH$. Then $g^{-1}h \in H$. Since $H$ is normal, for any $a \in G$ we have
\[
a^{-1}(g^{-1}h)a \in H.
\]
But
\[
a^{-1}(g^{-1}h)a = (ga)^{-1}(ha).
\]
Thus $(ga)^{-1}(ha) \in H$, which implies
\[
gaH = haH.
\]
This proves the claim.
\end{proof}

For normal subgroups, the approximations $C_H^G$ are not just close to the cofree functor, but are naturally isomorphic to the $H$-trace of the cofree functor. This provides a concrete representation of the trace that will be essential for further categorical constructions.

\begin{proposition}\label{prop:CHG_iso}
Let $G$ be a group and $H$ a normal subgroup of $G$. Then
\[
C_H^G \cong \operatorname{Tr}_H \circ C_G,
\]
as functors from $\mathbf{Set}$ to $G\text{-}\mathbf{Set}$.
\end{proposition}

\begin{proof}
For each set $X$, we define a map
\[
\theta_X \colon C_H^G(X) \longrightarrow \operatorname{Tr}_H(C_G(X))
\]
by
\[
\theta_X(\phi)(g) = \phi(gH),
\]
for $\phi \in C_H^G(X)$ and $g \in G$.

We first verify that $\theta_X(\phi)$ indeed lies in $\operatorname{Tr}_H(C_G(X))$. Let $\psi = \theta_X(\phi)$. Define a map
\[
\alpha \colon G/H \longrightarrow C_G(X), \qquad \alpha(aH) = a \psi,
\]
where the action on the right is the usual $G$-action on $C_G(X)$, i.e., $(a \psi)(g) = \psi(a^{-1}g)$.

We claim that $\alpha$ is well-defined. Suppose $aH = bH$. Then $a^{-1}b \in H$. For any $g \in G$, we have
\[
(a \psi)(g) = \psi(a^{-1}g) = \phi(a^{-1}gH).
\]
Since $H$ is normal, by Proposition~\ref{prop:normal_right_action}, the equality $aH = bH$ implies that $a^{-1}gH = b^{-1}gH$ for all $g \in G$. Hence
\[
\phi(a^{-1}gH) = \phi(b^{-1}gH) = \psi(b^{-1}g) = (b \psi)(g).
\]
Thus $a \psi = b \psi$, so $\alpha$ is well-defined. Moreover, $\alpha$ is $G$-equivariant because for any $x \in G$ and $aH \in G/H$,
\[
\alpha(xaH) = (xa) \psi = x (a \psi) = x \alpha(aH).
\]
Since $\alpha(H) = \psi$, we have $\psi \in \operatorname{im}(\alpha) \subseteq \operatorname{Tr}_H(C_G(X))$. Hence $\theta_X$ is well-defined.

Next, we prove that $\theta_X$ is $G$-equivariant. For $\phi \in C_H^G(X)$, $a, g \in G$,
\[
\theta_X(g \phi)(g) = (g \phi)(gH) = \phi(g^{-1}gH) = \theta_X(\phi)(g^{-1}g) = (g \theta_X(\phi))(g).
\]
Thus $\theta_X(g \phi) = g \theta_X(\phi)$.

We now show that $\theta_X$ is injective. If $\theta_X(\phi) = \theta_X(\psi)$, then for any $gH \in G/H$,
\[
\phi(gH) = \theta_X(\phi)(g) = \theta_X(\psi)(g) = \psi(gH),
\]
so $\phi = \psi$.

It remains to prove surjectivity. Let $\psi \in \operatorname{Tr}_H(C_G(X))$. By definition, there exists a $G$-equivariant map $\alpha \colon G/H \longrightarrow C_G(X)$ such that $\psi \in \operatorname{im}(\alpha)$. Pick $aH \in G/H$ with $\psi = \alpha(aH)$. Define
\[
\phi \colon G/H \longrightarrow X, \qquad \phi(gH) = \psi(g).
\]
We must show that $\phi$ is well-defined. Suppose $gH = hH$. Then, using the equivariance of $\alpha$,
\[
g \psi = g \alpha(aH) = \alpha(gaH),
\]
and similarly
\[
h \psi = \alpha(haH).
\]
By Proposition~\ref{prop:normal_right_action}, since $gH = hH$, we have $gaH = haH$. Hence $g \psi = h \psi$. Evaluating these equal functions at the identity element $e \in G$, we obtain
\[
\psi(g) = (g \psi)(e) = (h \psi)(e) = \psi(h).
\]
Thus $\phi(gH) = \phi(hH)$, so $\phi$ is well-defined. Finally, for any $g \in G$,
\[
\theta_X(\phi)(g) = \phi(gH) = \psi(g),
\]
so $\theta_X(\phi) = \psi$. Hence $\theta_X$ is surjective, and therefore bijective.

Finally, we verify that $\theta$ is a natural transformation. Let $f \colon X \longrightarrow Y$ be a map of sets. For $\phi \in C_H^G(X)$ and $g \in G$,
\[
\theta_Y(C_H^G(f)(\phi))(g)
= C_H^G(f)(\phi)(gH)
= f(\phi(gH))
= f(\theta_X(\phi)(g))
= C_G(f)(\theta_X(\phi))(g)
= \operatorname{Tr}_H(C_G(f))(\theta_X(\phi))(g),
\]
where the last equality holds because $\operatorname{Tr}_H$ is a functor (Proposition~\ref{prop:tr_preradical}) and restricts the image of $C_G(f)$ to the trace. Thus the required naturality square commutes.

Therefore, $\theta$ is a natural isomorphism between the functors $C_H^G$ and $\operatorname{Tr}_H \circ C_G$.
\end{proof}

We now combine all the previous results to prove the main theorem of this section: the generalized cofree functor $C_H^G$ is exactly the right adjoint to the forgetful functor on the category of $H$-generated $G$-sets.

\begin{proposition}\label{prop:CHG_adjoint}
Let $G$ be a group and $H$ a normal subgroup of $G$. Then the corestricted functor
\[
C_H^G \colon \mathbf{Set} \longrightarrow G\text{-}\mathbf{Set}_H
\]
(which exists by Proposition~\ref{prop:CHG_iso}) is the right adjoint of the underlying functor
\[
U_H \colon G\text{-}\mathbf{Set}_H \longrightarrow \mathbf{Set},
\]
which forgets the $G$-action.
\end{proposition}

\begin{proof}
Let
\[
U \colon G\text{-}\mathbf{Set} \longrightarrow \mathbf{Set}
\]
denote the usual forgetful functor, and let
\[
I_H \colon G\text{-}\mathbf{Set}_H \longrightarrow G\text{-}\mathbf{Set}
\]
be the inclusion functor of the full subcategory of $H$-generated $G$-sets.

By Proposition~\ref{prop:cofree_adjoint}, we have the adjunction
\[
U \dashv C_G,
\]
where $C_G \colon \mathbf{Set} \longrightarrow G\text{-}\mathbf{Set}$ is the cofree functor.

Furthermore, by Proposition~\ref{prop:tr_coreflector}, the subcategory $G\text{-}\mathbf{Set}_H$ is coreflective in $G\text{-}\mathbf{Set}$, with coreflector $\operatorname{Tr}_H$. Hence we have the adjunction
\[
I_H \dashv \operatorname{Tr}_H,
\]
where $\operatorname{Tr}_H \colon G\text{-}\mathbf{Set} \longrightarrow G\text{-}\mathbf{Set}_H$ is the $H$-trace functor.

Composing these two adjunctions, we obtain
\[
U \circ I_H \dashv \operatorname{Tr}_H \circ C_G.
\]

Now observe that $U \circ I_H = U_H$, since both functors first forget the $G$-action and then restrict to the subcategory. Therefore,
\[
U_H \dashv \operatorname{Tr}_H \circ C_G.
\]

Finally, by Proposition~\ref{prop:CHG_iso}, we have a natural isomorphism
\[
C_H^G \cong \operatorname{Tr}_H \circ C_G.
\]
Thus we may replace the right adjoint $\operatorname{Tr}_H \circ C_G$ by the naturally isomorphic functor $C_H^G$ (corestricted to $G\text{-}\mathbf{Set}_H$). Consequently,
\[
U_H \dashv C_H^G.
\]
This proves that $C_H^G$ is the right adjoint of the underlying functor $U_H$.
\end{proof}

\section{Comonads and the classification of groups}
\label{sec:comonads}

In this final section we study the comonad induced on the category of sets by the cofree adjunction. Our aim is to show that the assignment $G \mapsto K_G$ is a fully faithful contravariant functor from the category of groups to the category of comonads on $\mathbf{Set}$. In particular, the cofree comonad determines the group up to isomorphism, and the essential image forms a full subcategory equivalent to $\mathbf{Grp}^{\mathrm{op}}$.

We begin by recalling the standard construction. Let $\mathcal C$ and $\mathcal D$ be categories, and let $F \colon \mathcal C \rightarrow \mathcal D$ and $U \colon \mathcal D \rightarrow \mathcal C$ be functors with $F \dashv U$.  Denote the unit by $\eta \colon 1_{\mathcal C} \Rightarrow U F$ and the counit by $\varepsilon \colon F U \Rightarrow 1_{\mathcal D}$. Then the endofunctor $T := U F \colon \mathcal C \rightarrow \mathcal C$ carries a canonical comonad structure, called the \emph{comonad induced by the adjunction}. Its counit is simply $\epsilon := \varepsilon \colon T \Rightarrow 1_{\mathcal C}$, and its comultiplication is given by
\[
\delta := U \eta F \colon T \Rightarrow T^2,
\]
with components $\delta_X = U(\eta_{F(X)})$. The coassociativity and counit axioms follow directly from the triangle identities of the adjunction.

We now apply this general construction to the adjunction $U_G \dashv C_G$ studied in Section~\ref{sec:cofree}. The induced comonad on $\mathbf{Set}$ will be denoted by
\[
K_G := U_G C_G.
\]
For a set $X$, we have $K_G(X) = U_G(C_G(X)) = X^G$, the set of all functions from $G$ to $X$.

\begin{definition}
For a group $G$, the comonad $K_G = (K_G, \epsilon^G, \delta^G)$ on $\mathbf{Set}$ is defined as follows. The endofunctor $K_G$ sends a set $X$ to the set of functions $X^G$, and for a map $f: X \rightarrow Y$, it acts by post-composition: $K_G(f)(\phi) = f \circ \phi$ for each $\phi \in X^G$. The counit $\epsilon_X^G: X^G \rightarrow X$ is given by evaluation at the identity element of $G$, that is, $\epsilon_X^G(\phi) = \phi(e_G)$ for all $\phi \in X^G$. Finally, the comultiplication $\delta_X^G: X^G \rightarrow (X^G)^G$ is defined pointwise by
\[
\delta_X^G(\phi)(g)(h) = \phi(gh),
\]
for all $\phi \in X^G$ and all $g, h \in G$.
\end{definition}

We briefly derive the formula for the comultiplication. The unit of the adjunction $U_G \dashv C_G$ at a $G$-set $Y$ is
\[
\eta^C_Y \colon Y \rightarrow Y^G, \qquad \eta^C_Y(y)(g) = g^{-1}y,
\]
where the dot denotes the action of $G$ on $Y$. 
Taking $Y = C_G(X) = X^G$, we have for $\phi \in X^G$ and $g,h \in G$:
\[
\eta^C_{C_G(X)}(\phi)(g)(h) = (g^{-1}\phi)(h) = \phi((g^{-1})^{-1}h) = \phi(g h).
\]
Since $\delta^G_X = U_G(\eta^C_{C_G(X)})$, the displayed formula follows.
 
The assignment $G \mapsto K_G$ is not merely a correspondence on objects; it extends to a functor in a natural way.

\begin{proposition}\label{prop:functoriality_comonad}
The rule $G \mapsto K_G$ defines a functor
\[
\Phi \colon \mathbf{Grp}^{\mathrm{op}} \longrightarrow \mathbf{Comonads}(\mathbf{Set}).
\]
More precisely, for a group morphism $f \colon G \longrightarrow H$, the induced comonad morphism
\[
\Phi(f) \colon K_H \Longrightarrow K_G
\]
has components
\[
\Phi(f)_X(\alpha) = \alpha \circ f,
\qquad 
\alpha \in K_H(X) = X^H,\; X \in \mathbf{Set}.
\]
\end{proposition}

\begin{proof}
We first check that $\Phi(f)$ is a natural transformation. For any map $u \colon X \rightarrow Y$ and any $\alpha \in X^H$,
\[
K_G(u)(\Phi(f)_X(\alpha)) = u \circ (\alpha \circ f) = (u \circ \alpha) \circ f = \Phi(f)_Y(K_H(u)(\alpha)),
\]
so the naturality square commutes.

Next, we verify that $\Phi(f)$ respects the comonad structure. 
For the counit, take $\alpha \in X^H$:
\[
\epsilon^G_X(\Phi(f)_X(\alpha)) = (\alpha \circ f)(e_G) = \alpha(f(e_G)) = \alpha(e_H) = \epsilon^H_X(\alpha),
\]
because $f$ is a group morphism. For the comultiplication, we need
\[
\delta^G \circ \Phi(f) = \bigl( K_G(\Phi(f)) \circ \Phi(f)_{K_H} \bigr) \circ \delta^H.
\]
Fix $X$, $\alpha \in X^H$, and $g_1,g_2 \in G$. 
The left-hand side evaluated at these elements gives
\[
\delta^G_X(\Phi(f)_X(\alpha))(g_1)(g_2) = \Phi(f)_X(\alpha)(g_1 g_2) = \alpha(f(g_1 g_2)).
\]
On the right-hand side we obtain
\[
\delta^H_X(\alpha)(f(g_1))(f(g_2)) = \alpha(f(g_1) f(g_2)).
\]
Since $f$ is a morphism, $f(g_1 g_2) = f(g_1)f(g_2)$, and hence the two sides coincide. Therefore $\Phi(f)$ is a morphism of comonads. The functoriality axioms (preservation of identities and composition) follow immediately from the definitions.
\end{proof}

We now isolate the image of the functor $\Phi$.

\begin{definition}
Let $\mathbf{GrpComonads}$ denote the full subcategory of $\mathbf{Comonads}(\mathbf{Set})$ whose objects are comonads isomorphic to $K_G$ for some group $G$. We call this the \emph{subcategory of group comonads}.
\end{definition}

The main result of this section is that $\Phi$ induces an equivalence between $\mathbf{Grp}^{\mathrm{op}}$ and $\mathbf{GrpComonads}$. We first establish faithfulness.

\begin{proposition}\label{prop:faithfulness}
The functor $\Phi \colon \mathbf{Grp}^{\mathrm{op}} \rightarrow \mathbf{Comonads}(\mathbf{Set})$ is faithful.
\end{proposition}

\begin{proof}
Suppose $f_1, f_2 \colon G \rightarrow H$ are group morphisms such that $\Phi(f_1) = \Phi(f_2)$. 
Then for every set $X$ and every $\alpha \in X^H$, we have $\alpha \circ f_1 = \alpha \circ f_2$. 
Taking $X = H$ and $\alpha = \mathrm{id}_H$, we obtain $f_1 = f_2$. 
Thus $\Phi$ is faithful.
\end{proof}

The next proposition is the key step: it shows that every morphism of comonads between cofree comonads comes from a unique group morphism.

\begin{proposition}\label{prop:fullness}
The functor $\Phi$ is full. 
More precisely, if $\alpha \colon K_H \Longrightarrow K_G$ is any morphism of comonads, then there exists a unique group morphism $f \colon G \longrightarrow H$ such that $\alpha = \Phi(f)$.
\end{proposition}

\begin{proof}
For every set $X$, the component $\alpha_X \colon X^H \rightarrow X^G$ is a natural transformation between the representable functors $\operatorname{Hom}(-, H)$ and $\operatorname{Hom}(-, G)$. By the Yoneda lemma, there exists a unique function $\beta \colon G \rightarrow H$ such that, for every set $X$ and every $\psi \in X^H$,
\[
\alpha_X(\psi) = \psi \circ \beta.
\]
We shall prove that $\beta$ is a group morphism by showing that it preserves multiplication.

We use the comultiplication condition for comonad morphisms:
\[
\delta^G \circ \alpha = \bigl( K_G(\alpha) \circ \alpha_{K_H} \bigr) \circ \delta^H.
\]
We evaluate both sides at an arbitrary set $X$, a function $\psi \in X^H$, and elements $g_1, g_2 \in G$.

For the left-hand side, we have
\[
\delta^G_X(\alpha_X(\psi))(g_1)(g_2)
= \alpha_X(\psi)(g_1 g_2)
= \psi(\beta(g_1 g_2)).
\]

For the right-hand side, set $\phi = \alpha_{K_H(X)}(\delta^H_X(\psi)) \in K_G(K_H(X))$. 
For any $g \in G$, we get
\[
\phi(g) = \delta^H_X(\psi)(\beta(g)) \in X^H,
\]
which is the function $h \mapsto \psi(\beta(g) h)$. 
Now, applying $K_G(\alpha_X)$ to $\phi$ and evaluating at $g_1, g_2$ yields
\[
\begin{aligned}
\bigl(K_G(\alpha_X)(\phi)\bigr)(g_1)(g_2)
&= \alpha_X(\phi(g_1))(g_2) \\
&= \phi(g_1)(\beta(g_2)) \\
&= \delta^H_X(\psi)(\beta(g_1))(\beta(g_2)) \\
&= \psi(\beta(g_1)\beta(g_2)).
\end{aligned}
\]
Thus the right-hand side equals $\psi(\beta(g_1)\beta(g_2))$.

Comparing both expressions gives
\[
\psi(\beta(g_1 g_2)) = \psi(\beta(g_1)\beta(g_2))
\]
for all $X$, all $\psi \in X^H$, and all $g_1, g_2 \in G$. Taking $X = H$ and $\psi = \mathrm{id}_H$, we obtain the multiplicativity condition:
\[
\beta(g_1 g_2) = \beta(g_1)\beta(g_2). \tag{1}
\]

Hence $\beta$ is a group morphism $G \to H$. By construction, $\alpha = \Phi(\beta)$. Therefore, $\Phi$ is full.
\end{proof}

Combining the propositions above yields the announced equivalence.

\begin{corollary}\label{cor:full_embedding}
The functor $\Phi$ induces an equivalence of categories
\[
\mathbf{Grp}^{\mathrm{op}} \simeq \mathbf{GrpComonads}.
\]
In other words, the assignment $G \mapsto K_G$ is a fully faithful contravariant embedding of the category of groups into the category of comonads on $\mathbf{Set}$.
\end{corollary}

As an immediate consequence, the isomorphism problem for groups is faithfully encoded by their cofree comonads.

\begin{corollary}\label{cor:reconstruction}
Let $G$ and $H$ be groups. If the comonads $K_G$ and $K_H$ are isomorphic in $\mathbf{Comonads}(\mathbf{Set})$, then the groups $G$ and $H$ are isomorphic.
\end{corollary}

\begin{proof}
Suppose $\alpha \colon K_G \rightarrow K_H$ is a comonad isomorphism. 
By fullness, there exists a group morphism $f \colon H \rightarrow G$ such that $\alpha = \Phi(f)$. 
Since $\alpha$ is an isomorphism, its inverse $\alpha^{-1} \colon K_H \rightarrow K_G$ corresponds, again by fullness, to a group morphism $g \colon G \rightarrow H$ with $\Phi(g) = \alpha^{-1}$. 
Then
\[
\Phi(g \circ f) = \Phi(g) \circ \Phi(f) = \alpha^{-1} \circ \alpha = \mathrm{id}_{K_G} = \Phi(\mathrm{id}_G).
\]
Faithfulness implies $g \circ f = \mathrm{id}_G$; similarly, $f \circ g = \mathrm{id}_H$. 
Thus $f$ and $g$ are inverse isomorphisms, so $G \cong H$.
\end{proof}

This completes the classification theorem. The cofree comonad $K_G$ captures the group $G$ up to isomorphism, and the entire categorical structure of groups is reflected in the comonads on sets.

\bibliography{biblio}
\bibliographystyle{amsplain}

\end{document}